\documentclass[11pt]{amsart} \usepackage{latexsym, amssymb, stmaryrd,amsthm, amsmath, amscd}
\usepackage[T1]{fontenc}

\DeclareTextSymbol{\thh}{T1}{254}

\newtheorem{thm}{Theorem}[section]
\newtheorem{lemma}[thm]{Lemma}
\newtheorem{prop}[thm]{Proposition}
\newtheorem{cor}[thm]{Corollary}

\theoremstyle{definition}
\newtheorem{df}[thm]{Definition}

\newtheorem{ex}[thm]{Example}

\newcommand{\Z}{\mathbb{Z}}

\newcommand{\D}{\curly{D}}

\makeatletter

\def\indsym#1#2{%
  \setbox0=\hbox{$\m@th#1x$}%
  \kern\wd0%
  \hbox to 0pt{\hss$\m@th#1\mid$\hbox to 0pt{$\m@th#1^{#2}$}\hss}%
  \lower.9\ht0\hbox to 0pt{\hss$\m@th#1\smile$\hss}%
  \kern\wd0}

\def\nindsym#1#2{%
  \setbox0=\hbox{$\m@th#1x$}%
  \kern\wd0%
  \hbox to 0pt{\hss$\m@th#1\not$\kern1.4\wd0\hss}
  \hbox to 0pt{\hss$\m@th#1\mid$\hbox to 0pt{$\m@th#1^{\,#2}$}\hss}%
  \lower.9\ht0\hbox to 0pt{\hss$\m@th#1\smile$\hss}%
  \kern\wd0}

\def\dotminussym#1#2{%
  \setbox0=\hbox{$\m@th#1-$}%
  \kern.5\wd0%
  \hbox to 0pt{\hss\hbox{$\m@th#1-$}\hss}%
  \raise.6\ht0\hbox to 0pt{\hss$\m@th#1.$\hss}%
  \kern.5\wd0}

\def \r { {\mathbb R} }
\def \<{\langle}
\def \>{\rangle}
\def \n {\mathbb N}

\def \*Z {{{^*}\Z}}

\def \((  {(\!(}
\def \)) {)\!)}

\def \dcl{\operatorname{dcl}}

\def \int{\operatorname{int}}
\numberwithin{equation}{section}

\def \ker{\operatorname{Ker}}
\def \coker{\operatorname{Coker}}

\def \sp{\overline{\operatorname{sp}}}
\def \span{\operatorname{span}}
\def \H{\mathbb{H}}
\def \D{\mathfrak{D}}
\def \c{\mathbb{C}}
\def \Re{\mathfrak{Re}}
\def \Im{\mathfrak{Im}}

\allowdisplaybreaks[2]

\begin{document}

\title{Definable Operators on Hilbert Spaces}

\author{Isaac Goldbring}

\address {University of California, Los Angeles, Department of Mathematics, 520 Portola Plaza, Box 951555, Los Angeles, CA 90095-1555, USA}
\email{isaac@math.ucla.edu}
\urladdr{www.math.ucla.edu/~isaac}

\begin{abstract}
Let $H$ be an infinite-dimensional (real or complex) Hilbert space, viewed as a metric structure in its natural signature.  We characterize the definable linear operators on $H$ as exactly the ``scalar plus compact'' operators.
\end{abstract}
\maketitle

\section{Introduction}

The continuous theory of infinite-dimensional (real) Hilbert spaces, denoted $\operatorname{IHS}$ in \cite{BBHU}, is one of the most well-understood theories in continuous logic.  For example, $\operatorname{IHS}$ admits quantifier elimination, is $\kappa$-categorical for every infinite cardinal $\kappa$, and is $\omega$-stable; moreover, one can identify the relation of nonforking independence concretely in terms of orthogonality of vectors.  In addition, one can completely understand the definable closure relation and the natural metric on the type spaces.  (See Section 15 of \cite{BBHU} for a more thorough discussion of the theory $\operatorname{IHS}$.)  However, there has yet to be any mention of what the definable sets or functions are in this theory.  In fact, there had yet to be any real study of definable functions in any metric structure until the paper \cite{Gold} analyzed the definable functions in the Urysohn sphere.

In this paper, we only study the definable \emph{linear operators} on Hilbert spaces, for studying arbitrary definable functions seems a bit out of reach at the moment.  As in \cite{Gold}, the key observation is the following:  If $\mathcal{M}$ is a metric structure, $A\subseteq M$ is a parameterset, and $f:M\to M$ is an $A$-definable function, then for every $x\in M$, we have $f(x)\in \dcl(Ax)$, where $\dcl$ stands for \emph{definable closure}.  Thus, in any theory where $\dcl$ is well-understood, one can begin to understand the definable functions.  In models of $\operatorname{IHS}$, the definable closure of a parameterset is equal to its closed linear span; see Lemma 15.3 of \cite{BBHU}.

Our main result is the following:  Let $H$ be an infinite-dimensional real (resp. complex) Hilbert space.  Then the definable linear operators on $H$ are exactly the ``scalar plus compact'' operators $\lambda I+K$, where $\lambda\in \r$ (resp. $\lambda \in \c$), $I:H\to H$ is the identity operator, and $K:H\to H$ is a compact operator.  In particular, this shows that many ``intuitively definable'' bounded linear operators on $H$ are not actually definable; for example, the left- and right-shift operators on $\ell^2$ are not definable.  This result shows that there can be an inherent gap between what is ``intuitively'' definable in a metric structure and what is actually definable.  (This is in contrast to \cite{Gold}, where many functions on the Urysohn sphere are shown to be non-definable; these functions are constructed in such a way that they do not appear to be definable in any reasonable sense of the word, whence the intuition and the logic agree.)  Another consequence of our main theorem is that the definable linear operators are closed under taking adjoints.

We also introduce a natural signature for complex Hilbert spaces and show that the characterization of definable linear operators as exactly the scalar plus compact operators persists in this context as well.  Since there are a few more structural results specific to operators on complex Hilbert spaces, our characterization of definable operators yields some extra corollaries in the complex situation, most notably the fact that the invariant subspace problem has a positive solution when restricted to definable operators.

On a side note, one should mention that the class of ``scalar plus compact'' operators has shown up in the recent work of Argyros-Haydon \cite{AH} where Banach spaces $X$ are constructed so that the only bounded linear operators on $X$ are the ``scalar plus compact'' operators.  According to Gowers' blog \cite{Gowers}, ``the Argyros-Haydon space has very definitely taken over as the new `nastiest known Banach space', in a sense that it has almost no non-trivial structure.'' 

We assume that the reader is familiar with the basics of continuous logic.  For the reader unacquainted with continuous logic, the survey \cite{BBHU} is the natural place to start.

I would like to thank Alex Berenstein and Christian Rosendal for useful discussions concerning this work.

\section{Preliminaries}

In this section, we let $H$ be an arbitrary infinite-dimensional real Hilbert space, viewed as a metric structure in the natural many-sorted language for Hilbert spaces, which we now briefly recall for the convenience of the reader.  For each $n\geq 1$, we have a sort for $B_n(H):=\{x\in H \ | \ \|x\|\leq n\}$.  For each $1\leq m\leq n$, we have a function symbol $I_{m,n}:B_m(H)\to B_n(H)$ for the inclusion mapping.  We also have, for each $n\geq 1$, the following symbols:
\begin{itemize}
\item function symbols $+,-:B_n(H)\times B_n(H)\to B_{2n}(H)$;
\item function symbols $r\cdot:B_n(H)\to B_{kn}(H)$ for all $r\in \r$, where $k$ is the unique natural number satisfying $k-1\leq |r|<k$;
\item a predicate symbol $\langle \cdot,\cdot \rangle: B_n(H)\times B_n(H)\to [-n^2,n^2]$;
\item a predicate symbol $\|\cdot \|:B_n(H)\to [0,n]$.
\end{itemize}


Observe that adding the norm as a predicate symbol is not altogether necessary since the norm is given by a quantifier-free formula using the inner product.  Finally, the metric on each sort is given by $d(x,y):=\|x-y\|$.  

Normally, the notion of a definable function is defined for functions from a product of sorts to another sort.  Thus, we must say exactly what we mean by a definable function $f:H\to H$.

\begin{df}\label{def}
Let $A\subseteq H$.  We say that a function $f:H\to H$ is \emph{$A$-definable} if:
\begin{enumerate}
\item for each $n\geq 1$, $f(B_n(H))$ is bounded; in this case, we let $m(n,f)\in \n$ be the minimal $m$ such that $f(B_n(H))$ is contained in $B_m(H)$;
\item for each $n\geq 1$ and each $m\geq m(n,f)$, the function $$f_{n,m}:B_n(H)\to B_m(H), \quad f_{n,m}(x)=f(x)$$ is $A$-definable, that is, the predicate $P_{n,m}:B_n(H)\times B_m(H)\to [0,1]$ defined by $P_{n,m}(x,y)=d(f(x),y)$ is $A$-definable. 
\end{enumerate} 
\end{df}


\noindent Observe that, since each $f_{n,m}$ can be defined using only countably many elements of $A$, a definable function $H\to H$ is always definable using only countably many parameters.  We will also need the following basic facts about definable functions:

\begin{lemma}\label{algebra}
If $f_1,f_2:H\to H$ are $A$-definable and $r\in \r$, then:
\begin{enumerate}
\item $r\cdot f_1$ is $A$-definable;
\item $f_1+f_2$ is $A$-definable;
\item $f_2\circ f_1$ is $A$-definable.
\end{enumerate}
\end{lemma}

\begin{proof}
(1) Without loss of generality, we may suppose that $r\not=0$.  Fix $n\geq 1$ and $m\geq m(n,r\cdot f_1)$.  Fix $x$ a variable of sort $B_n(H)$ and $y$ a variable of sort $B_m(H)$.  Let $k$ be the unique natural number such that $k-1\leq \frac{1}{|r|}<k$.  Let $Q:B_n(H)\times B_{km}(H)\to [0,1]$ be the $A$-definable predicate $Q(x,z)=\|f_1(x)-z\|$.  Then $\|(r\cdot f_1)(x)-y\|=|r|\cdot Q(x,\frac{1}{r}\cdot y)$, which is an $A$-definable predicate.  

\

\noindent (2)  Fix $n\geq 1$ and $m\geq m(n,f_1+f_2)$.  Fix $x$ a variable of sort $B_n(H)$ and $y$ a variable of sort $B_m(H)$.  Set $m':=\max(m,m(n,f_1),m(n,f_2))$.  Let $Q':B_n(H)\times B_{2m'}(H)\to [0,1]$ be the $A$-definable predicate $Q'(x,z)=\|f_1(x)-z\|$.  Then we have
$$\|(f_1+f_2)(x)-y\|=Q'(x,I_{m,m'}(y)-I_{m(n,f_2),m'}(f_2(x))),$$ which is an $A$-definable predicate since $f_2$ is an $A$-definable function.

\

\noindent (3)  One can just adapt the proof of this fact from 1-sorted continuous logic, keeping track of the sorts of variables as in the first two parts of the proof.
\end{proof}

It is evident from the proof of the above theorem that keeping track of which sorts various terms lie in can become quite cumbersome.  Thus, in the rest of this paper, we reserve the right to become a bit looser in this regards.

In the rest of this section, we fix $A\subseteq H$ and let $P:H\to H$ denote the orthogonal projection map onto $\sp(A)$; here, and in the rest of this paper, $\sp$ denotes closed linear span.
 
\begin{lemma}
Given $x\in H$, we have that $\sp(A\cup\{x\})=\sp(A)\oplus \r\cdot (x-Px)$.
\end{lemma}

\begin{proof}
The inclusion $\supseteq$ is clear.  We now prove the inclusion $\subseteq$.  We may suppose that $Px\not=x$.  Now suppose that $z\in \sp(A\cup \{x\})$, so $z=\lim z_n$, where $z_n\in \span(A\cup \{x\})$.  Write $z_n=y_n+\lambda_nx$, where $y_n\in \span(A)$ and $\lambda_n\in \r$.  Then $z_n=(y_n+\lambda_nPx)+\lambda_n(x-Px).$  Set $w_n:=y_n+\lambda_nPx\in \sp(A).$  Now
$$\|z_m-z_n\|^2=\|w_m-w_n\|^2+|\lambda_m-\lambda_n|^2\|x-Px\|^2,$$ so $w_n\to w\in \sp(A)$ and $\lambda_n\to \lambda\in \r$.  It follows that $$z=w+\lambda(x-Px)\in \sp(A)\oplus \r\cdot (x-Px).$$
\end{proof}

\begin{cor}
Suppose that $f:H\to H$ is $A$-definable and $x\in H$.  Then $f(x)\in \sp(A)\oplus \r\cdot (x-Px)$.
\end{cor}

\begin{proof}
This follows from the fact that $\dcl(B)=\sp(B)$ for any $B\subseteq H$.
\end{proof}

Suppose that $\H$ is an elementary extension of $H$.  Suppose that $f:H\to H$ is an $A$-definable function.  Fix $n\geq 1$ and $m\geq m(n,f)$.  By Proposition of 9.25 of \cite{BBHU},  there is a natural extension of $f_{n,m}$ to an $A$-definable function $f_{n,m}:B_n(\H)\to B_m(\H)$.  Moreover, by elementarity, we see that if $n'\geq n$, $m\geq m(n,f)$, $m'\geq m(n',f)$ and $x\in B_n(\H)$, then $f_{n,m}(x)=f_{n',m'}(x)$, whence the $f_{n,m}$'s piece together to yield an $A$-definable function $f:\H\to\H$.

\section{Definable Operators on Real Hilbert Spaces}

\noindent In this section, we continue to let $H$ be an infinite-dimensional real Hilbert space.  We aim to prove the following:

\begin{thm}\label{main}
Suppose that $T:H\to H$ is a bounded linear map.  Then $T$ is definable if and only if there is $\lambda\in \r$ and a compact operator $K:H\to H$ such that $T=\lambda I+K$.
\end{thm}

We can rephrase this theorem as follows.  Let $\D(H)$ denote the algebra of definable linear operators on $H$.  Let $\mathfrak{B}(H)$ denote the Banach algebra of bounded linear operators on $H$ and let $\mathfrak{B}_0(H)$ denote the closed, two-sided ideal of $\mathfrak{B}(H)$ consisting of the compact operators on $H$.  Finally, let $\mathfrak{C}(H)=\mathfrak{B}(H)/\mathfrak{B}_0(H)$ denote the Calkin algebra of $H$ with quotient map $\pi:\mathfrak{B}(H)\to \mathfrak{C}(H)$.  If $e$ is the unit element of $\mathfrak{C}(H)$, then we view $\r$ as a subalgebra of $\mathfrak{C}(H)$ by identifying it with $\r\cdot e$.  Then Theorem \ref{main} states that $\D(H)=\pi^{-1}(\r)$. 

\

\noindent We first prove the ``if'' direction of Theorem \ref{main}.
\begin{prop}\label{cpt}
Suppose that $T:H\to H$ is a linear operator on $H$.
\begin{enumerate}
\item If $T$ is a finite-rank operator, then $T$ is definable.  In fact, $d(T(x),y)$ is given by a formula.  
\item If $T$ is a compact operator, then $T$ is definable.
\end{enumerate}
\end{prop}

\begin{proof}
(1)  Suppose that $e_1,\ldots,e_n$ is an orthonormal basis for $T(H)$.  Then there exist bounded linear functionals $f_1,\ldots,f_n:H\to \r$ so that $$T(x)=f_1(x)e_1+\cdots +f_n(x)e_n$$ for all $x\in H$.  For each $i\in \{1,\ldots,n\}$, let $z_i\in H$ be the unique vector so that $f_i(x)=\langle x,z_i\rangle$ for all $x\in H$; this is possible by the Riesz Representation Theorem (see \cite{Con} Prop I.3.4).  Then $T(x)=\sum_{i=1}^n\langle x,z_i\rangle e_i$, whence, for $y\in H$, we have
$$d(T(x),y)=\sqrt{\sum_{i=1}^n (\langle x,z_i\rangle^2)-2\sum_{i=1}^n (\langle x,z_i\rangle \langle e_i,y\rangle)+\|y\|^2}.$$  

For (2), let $T$ be a compact operator and let $(T_n)$ be a sequence of finite-rank operators such that $\|T-T_n\|\to 0$; see, for example, \cite{Con} II.4.4.  Given $\epsilon >0$ and $n>0$, choose $N$ such that $\|T-T_N\|<\frac{\epsilon}{n}$.  Fix $m\geq m(n,T)$ and let $x$ and $y$ range over $B_n(H)$ and $B_m(H)$ respectively.  We then have
$$\big| \|T(x)-y\|-\|T_N(x)-y\|\big|\leq \|T(x)-T_N(x)\|<\epsilon.$$

\noindent Since $\|T_N(x)-y\|$ is given by a formula, we have that $\|T(x)-y\|$ is given by a definable predicate.
\end{proof}

Since $\lambda I$ is a definable linear map for every $\lambda\in \r$, the preceding proposition implies that $\lambda I+K$ is definable for every $\lambda\in \r$ and every $K\in \mathfrak{B}_0(H)$.  

We now aim to prove the ``only if'' direction of Theorem \ref{main}.  Until otherwise stated, we suppose that $T:H\to H$ is an $A$-definable linear operator, where $A\subseteq H$ is countable.  Furthermore, we fix a proper $\omega_1$-saturated elementary extension $\H$ of $H$ and we consider $T:\H\to \H$, the natural extension of $T$ to $\H$ as described at the end of the previous section.

\begin{lemma}
$T:\H\to \H$ is also linear.
\end{lemma}

\begin{proof}
Fix $n\geq 1$ and set $m:=m(2n,T)$.  Let $(\varphi_k(x,y))$ be a sequence of formulae with parameters from $A$ such that, for all $x\in B_{2n}(H)$ and $y\in B_m(H)$ we have $|d(T(x),y)-\varphi_k(x,y)|\leq \frac{1}{k}$.  Then $$H\models \sup_{x,y,\in B_n(H)}\sup_{z,w_1,w_2\in B_m(H)}(\max(\varphi_k(x+y,z),\varphi_k(x,w_1),\varphi_k(x,w_2))\leq \frac{1}{k}$$ $$\Rightarrow d(z,w_1+w_2)\leq \frac{6}{k})).$$  By Proposition 7.14 of \cite{BBHU}, this implication is true in $\H$.  It follows that $T(x+y)=T(x)+T(y)$ for all $x,y\in \H$.  A similar argument proves that $T$ preserves scalar multiplication.   
\end{proof}

As in the previous section, we let $P:\H\to \H$ denote the orthogonal projection onto $\sp(A)$.

\begin{prop}\label{lambda}
There exists a unique $\lambda \in \r$ such that $T=P\circ T+\lambda I-\lambda P$.
\end{prop}

\begin{proof}
First suppose that $x\in \sp(A)^\perp\subseteq \H$.  Then $T(x)-P(T(x))\in \r\cdot x$.  Suppose further that $y\in \sp(A)^\perp$.  Then there exist constants $\lambda_1,\lambda_2,\lambda_3\in \r$ such that $T(x)=P(T(x))+\lambda_1x$, $T(y)=P(T(y))+\lambda_2y$, and $T(x+y)=P(T(x+y))+\lambda_3(x+y)$.  From this we gather that $\lambda_1 x+\lambda_2 y=\lambda_3(x+y)$. It follows that if $x,y\not=0$, then $\lambda_1=\lambda_2$.  
Observe that, by $\omega_1$-saturation, $\sp(A)^\perp\not=\{0\}$.  Thus, there is a unique $\lambda\in \r$ such that, for all $x\in \sp(A)^\perp$, $T(x)=P(T(x))+\lambda x$.  Fix this $\lambda$ and suppose that $x\in \H$ is arbitrary.  Then $$T(x)=T(Px)+T(x-Px)=T(Px)+PT(x-Px)+\lambda(x-Px).$$  Since $Px\in \sp(A)$, we have $P(T(x))=T(Px)+PT(x-Px)$ and thus $T(x)=PT(x)+\lambda(x-Px)$.        
\end{proof}

\noindent From now on, we write $\lambda(T)$ for the unique $\lambda$ for which $T=P\circ T+\lambda I-\lambda P$.

\begin{prop}
$T-\lambda(T) I$ is a compact operator.
\end{prop}

\begin{proof}
Set $\lambda:=\lambda(T)$.  Observe that $T-\lambda I=P\circ (T-\lambda I)$, whence $(T-\lambda I)(\H)\subseteq \sp(A)$.  Since $\H$ is $\omega_1$-saturated, we know that $(T-\lambda I)(B_1(\H))$ is closed.  We thus need to show that $(T-\lambda I)(B_1(\H))$ is compact.  Let $\epsilon>0$ be given.  Set $m:=m(1,T)$.  Let $(a_n)$ be a countable dense subset of $(T-\lambda I)(B_1(\H))$.  Let $k:=\max(|\lambda|,m)$.  Let $x$ range over variables of sort $B_1(\H)$ and $y$ range over variables of sort $B_{2k}(H)$.  Let $\varphi(x,y)$ be a formula such that $\big|\|T(x)-y\|-\varphi(x,y)\big|<\frac{\epsilon}{4}$.  Then the following set of formulae is inconsistent:
$$\{\varphi(x,\lambda x+a_n)\geq \frac{\epsilon}{2} \ | \ n\in \n\}.$$  By saturation, there are $a_1,\ldots,a_n$ such that $a_1,\ldots,a_n$ form an $\epsilon$-net for $(T-\lambda I)(B_1(\H))$.
\end{proof}

This finishes the proof of Theorem \ref{main}.  Let us now consider some of its consequences.

\begin{cor}\label{cstar}
$\D(H)$ is a $C^*$-subalgebra of $\mathfrak{B}(H)$.
\end{cor}

The preceding corollary is interesting because it is not at all clear, from first principles, that $\D(H)$ is closed under taking adjoints.  However, it is easy to see that the adjoint of a definable \emph{normal} operator $T:H\to H$ is definable, for we then have
$$\|T^*(x)-y\|^2=\|T^*(x)\|^2-2\langle T^*(x),y\rangle +\|y\|^2=\|T(x)\|^2-2\langle T(y),x\rangle +\|y\|^2,$$ which is a definable predicate since $T$ is definable.

\begin{cor}\label{fin}
Suppose that $T\in \D(H)$ is not compact.  Then $\ker(T)$ and $\coker(T)$ are finite-dimensional.  Moreover, $\ker(T)\subseteq \sp(A)$.
\end{cor}

\begin{proof}
Suppose that $T$ is not a compact operator.  Then $\lambda(T)\not=0$ and we have, by Proposition \ref{lambda}, that $\ker(T)\subseteq \sp(A)$.  Let $m:=m(1,T)$.  Let $x$ and $y$ range over $B_1(\H)$ and $B_m(\H)$ respectively.  For each $k\geq 1$, let $\varphi_k(x,y)$ be a formula such that $|d(T(x),y)-\varphi_k(x,y)|<\frac{1}{k}$ for all $x$ and $y$.  Let $(a_i)$ be a countable dense subset of $\sp(A)$.  Fix $\epsilon>0$.  Then the set of conditions
$$\{\varphi_k(x,0)\leq \frac{1}{k} \ | \ k\geq 1\}\cup \{d(x,a_i)\geq \epsilon \ | \ i\geq 1\}$$ is unsatisfiable.  By $\omega_1$-saturation, there are $a_1,\ldots,a_k$ which form an $\epsilon$-net for the unit ball of $\ker(T)$.  Since $\epsilon>0$ was arbitrary, this shows that the unit ball of $\ker(T)$ is compact, whence $\ker(T)$ is finite-dimensional.  Since $T^*$ is also definable, we see that $\ker(T^*)$ is also finite-dimensional, implying that $\coker(T)$ is finite-dimensional.
\end{proof}






\begin{cor}\label{proj}
Suppose that $K$ is a closed subspace of $H$ and $T:H\to H$ is the orthogonal projection onto $K$.  Then $T$ is definable if and only if $K$ is of finite dimension or finite codimension.
\end{cor}

\begin{proof}
If $K$ is of finite dimension or finite codimension, then $T$ or $I-T$ is finite-rank, whence definable.  Conversely, suppose that $T$ is definable.  If $T$ is compact, then $T$ is finite-rank (as it is idempotent), whence $K$ is finite-dimensional.  Otherwise, by Corollary \ref{fin}, we have $$\dim(H/K)=\dim(K^\perp)=\dim(\ker(T))<\infty.$$
\end{proof}

\noindent In this paper, we let $\ell^2_\r$ (resp. $\ell^2_\c$) denote the real (resp. complex) Hilbert space of all real (resp. complex) square-summable sequences indexed by $\n$.

\begin{cor}\label{l2}
Let $I=\{i_1,i_2,\ldots,\}$ be an infinite and co-infinite subset of $\n$ and let $T:\ell^2_\r\to \ell^2_\r$ be defined by $T(x)_n=x_{i_n}$.  Then $T$ is not definable.
\end{cor}

\begin{proof}
Observe that $T(B_1(\ell^2_\r))=B_1(\ell^2_\r)$, so $T$ is not a compact operator.  Since $\ker(T)$ is infinite-dimensional, $T$ cannot be definable by Corollary \ref{fin}.
\end{proof}



\begin{cor}\label{eigen}
Suppose that $T:H\to H$ is a definable linear operator and $\mu$ is an eigenvalue of $T$ satisfying $\mu\not=\lambda(T)$.  Then the eigenspace $E_\mu(T)$ is a finite-dimensional subspace of $\sp(A)$.
\end{cor}

\begin{proof}
Set $\lambda:=\lambda(T)$.  Fix $\mu\not=\lambda$ and suppose that $z\not=0$ is such that $T(z)=\mu z$.  We know that $T(z)=P(T(z))+\lambda(z-Pz)$.  Thus $$(\mu-\lambda) z=P(T(z))-\lambda Pz\in \sp(A),$$ whence $z\in \sp(A)$.  Thus $E_\mu(T)$ is contained in $\sp(A)$.  
Now observe that $\mu-\lambda$ is a nonzero eigenvalue of $T-\lambda I$; since $T-\lambda I$ is compact, $E_{\mu-\lambda}(T-\lambda I)$ is finite-dimensional by the Spectral Theorem for Compact Operators (see \cite{Con}, VII.7.1).  Now use the fact that $E_\mu(T)=E_{\mu-\lambda}(T-\lambda I)$.
\end{proof}

In particular, if $T:H\to H$ is an $A$-definable linear operator, where $\sp(A)$ is finite-dimensional, then $T$ has only finitely many eigenvalues.

\section{Definable Operators on Complex Hilbert Spaces}

In this section, we let $H$ be an infinite-dimensional \emph{complex} Hilbert space.  We treat $H$ as a metric structure just as in the case of real Hilbert spaces except for two important differences.  First, in addition to all of the function symbols for scalar multiplication by real numbers, we include, for each $n\geq 1$, a function symbol $i\cdot:B_n(H)\to B_n(H)$ for scalar multiplication by $i$.  Secondly, for each $n\geq 1$, we replace the predicate symbol for the inner product by two predicate symbols $\Re,\Im:B_n(H)\times B_n(H)\to [-n^2,n^2]$, which are to be interpreted as the real and imaginary parts of the inner product.  

In this signature, it is still true that definable closure in $H$ coincides with closed linear span in $H$.  Moreover, it is straightforward to verify that all of the results from Section 2 as well as all of the results leading up to the proof of Theorem \ref{main} remain true in the complex context.  For example, consider the finite-rank operator $T:H\to H$ given by $T(x)=\sum_{i=1}^n \langle x,z_i\rangle e_i$, where $\{e_1,\ldots,e_n\}$ is an orthonormal set in $H$ and $z_1,\ldots,z_n\in H$ are arbitrary.  Then we have
$$d(T(x),y)=\sqrt{\sum_{i=1}^n \big(|\langle x,z_i\rangle| ^2-\langle x,z_i\rangle \langle e_i,y\rangle -\langle z_i,x\rangle \langle y,e_i\rangle\big)+\|y\|^2}.$$  Now $|\langle x,z_i\rangle|^2=\Re(x,z_i)^2+\Im(x,z_i)^2$ and $$\langle x,z_i\rangle \langle e_i,y\rangle+\langle z_i,x\rangle \langle y,e_i\rangle=2(\Re(x,z_i)\Re(e_i,y)-\Im(x,z_i)\Im(e_i,y)).$$  It thus follows that $d(T(x),y)$ is once again given by a formula.  Performing similar modifications to the rest of the above arguments yields a complex version of our main theorem:

\begin{thm}
A bounded linear operator $T:H\to H$ is definable if and only if there exists $\lambda\in \mathbb{C}$ and a compact operator $K:H\to H$ such that $T=\lambda I+K$.
\end{thm}



We once again write $\D(H)$ for the algebra of definable operators.  Observe that we have complex versions of Corollaries \ref{cstar} through \ref{eigen}.  In addition, in the complex context, we may draw a few more conclusions from our result on definable operators, which we discuss now.  

Recall that a bounded operator $T:H\to H$ is said to be \emph{Fredholm} if both $\ker(T)$ and $\coker(T)$ are finite-dimensional.  If $T$ is Fredholm, then the \emph{index} of $T$ is the integer $\operatorname{ind}(T):=\dim \ker(T)-\dim \coker(T).$

\begin{cor}\label{Fred}
If $T\in \D(H)$, then either $T$ is compact or else $T$ is a Fredholm operator of index $0$.  In the latter case, we have that $\ker(T)$ is a finite-dimensional subspace of $\sp(A)$.
\end{cor}

\begin{proof}
The first statement follows from the Fredholm alternative from functional analysis; see \cite{Con}, VII.7.9 and XI.3.3.  If $T$ is Fredholm, then the fact that $\ker(T)\subseteq \sp(A)$ follows directly from Proposition \ref{lambda}.
\end{proof}

\noindent Let $\mathbb{F}$ denote either $\r$ or $\c$.  Recall the \emph{left-} and \emph{right-shift operators} $L_\mathbb{F}$ and $R_\mathbb{F}$ on $\ell^2_\mathbb{F}$:
$$L_\mathbb{F}:\ell^2_\mathbb{F}\to\ell^2_\mathbb{F}, \quad L_\mathbb{F}(x_0,x_1,x_2,\ldots)=(x_1,x_2,x_3,\ldots),$$
$$R_\mathbb{F}:\ell^2_\mathbb{F}\to \ell^2_\mathbb{F}, \quad R_\mathbb{F}(x_0,x_1,x_2,\ldots)=(0,x_0,x_1,x_2,\ldots).$$  
\begin{cor}
The left and right shift operators $L_\c,R_\c:\ell^2_\c\to \ell^2_\c$ are not definable.  Consequently, the left- and right-shift operators $L_\r,R_\r:\ell^2_\r\to \ell^2_\r$ are not definable. 
\end{cor}

\begin{proof}
$L_\c$ and $R_\c$ are Fredholm operators of index $1$ and $-1$ respectively, whence not definable.  If $L_\r$ were definable, then there would be $\lambda \in \r$ and a compact operator $K:\ell^2_\r\to \ell^2_\r$ such that $L_\r=\lambda I+K$.  Let $K^\c$ denote the canonical extension of $K$ to a $\c$-linear map on $\ell^2_\c$; observe that $K^\c$ is a compact operator.  Then $L_\c=\lambda I+K^\c$, which is a scalar plus compact operator on $\ell^2_\c$, implying that $L_\c$ is definable, a contradiction.  The same reasoning shows that $R_\r$ is not definable.
\end{proof}

\noindent As above, we let $\mathfrak{C}(H)$ denote the Calkin algebra of $H$ with identity element $e$ and we let $\pi:\mathfrak{B}(H)\to \mathfrak{C}(H)$ denote the canonical quotient map onto the Calkin algebra of $H$.  Given $T\in \mathfrak{B}(H)$, recall that the \emph{essential spectrum} of $T$ is $$\sigma_e(T):=\{\lambda \in \c \ | \ \pi(T)-\lambda e \text{ is not invertible}\}.$$  The following result is clear from our main theorem.

\begin{cor}\label{ess}
If $T\in \D(H)$, then $\sigma_e(T)=\{\lambda(T)\}$.
\end{cor}
 
\begin{ex}
Consider the operator $L_\c\oplus R_\c\in \mathfrak{B}(\ell^2_\c\oplus \ell^2_\c)$.  Then $L_\c\oplus R_\c$ is Fredholm of index $0$ by XI.2.2 and X1.3.10 of \cite{Con}.  Thus, Corollary \ref{Fred} does not rule out the possibility that $L_\c\oplus R_\c$ is definable.  However, XI.4.11 of \cite{Con} shows that $\sigma_e(L_\c\oplus R_\c)=\{z\in \c \ | \ |z|=1\}$, whence Corollary \ref{ess} shows that $L_\c\oplus R_\c$ is not definable.
\end{ex}

Recall the \emph{invariant subspace problem for Hilbert spaces}:  Let $H$ be the \emph{separable} complex Hilbert space.  Given $T\in \mathfrak{B}(H)$, does there exist a nontrivial closed subspace $E$ of $H$ such that $T(E)\subseteq E$?  Here, by a nontrivial subspace of $H$, we mean a subspace of $H$ other than $\{0\}$ and $H$.  While this problem remains open, we do know that the answer is positive if one restricts attention to definable bounded operators:

\begin{cor}
Suppose that $H$ is the separable complex Hilbert space.  Then given any $T\in \D(H)$, there is a nontrivial closed subspace $E$ of $H$ such that $T(E)\subseteq E$.
\end{cor}

\begin{proof}
Write $T=\lambda I+K$, where $\lambda\in \c$ and $K\in \mathfrak{B}_0(H)$.  If $K=\{0\}$, then take $E:=\c\cdot x$, where $x\in H\setminus \{0\}$ is arbitrary.  Otherwise, observe that $T$ commutes with $K$; combine this with the result of Lomonosov (see \cite{Con}, VI.4.13) which states that a bounded linear operator on a complex Banach space which commutes with a nonzero compact operator must have a proper closed invariant subspace.
\end{proof}








\begin{thebibliography}{1}
\bibitem{AH} S. Argyros and R. Haydon, \textit{A hereditarily indecomposable $L_\infty$-space that solves the scalar-plus-compact problem}, arXiv 0903.3921v2.
\bibitem{BBHU} I. Ben Yaacov, A. Berenstein, C. W. Henson, A. Usvyatsov, \textit{Model theory for metric structures}, Model theory with applications to algebra and analysis. Vol. 2, pgs. 315-427, London Math. Soc. Lecture Note Ser. (350), Cambridge Univ. Press, Cambridge, 2008.
\bibitem{Con} J. Conway, \textit{A Course in Functional Analysis, Second Edition}, Graduate Texts in Mathematics \textbf{96}, Springer, 1990.
\bibitem{Gold} I. Goldbring, \textit{Definable functions in Urysohn's metric space}, submitted.  arXiv 1001.4999
\bibitem{Gowers} T. Gowers, \texttt{http://gowers.wordpress.com/2009/02/07}
\end{thebibliography}
\end{document}